\documentclass[12pt]{amsart}

\usepackage[english]{babel}

\usepackage{amssymb}
\usepackage{mathrsfs}
\usepackage{hyperref}
\usepackage{xcolor}
\usepackage{eucal}

\usepackage{amssymb}
\usepackage{amsmath}
\usepackage{amsthm}
\usepackage{imakeidx}
\usepackage{breqn}

\usepackage{diagbox}

\usepackage{verbatim}

\usepackage{fancyhdr}

\usepackage{tikz-cd}

\usepackage[utf8]{inputenc}

\usepackage{graphicx}

\usepackage{xcolor}

\theoremstyle{definition}

\newtheorem{mydef}{Definition}[section]

\newtheorem{myque}[mydef]{Question}

\theoremstyle{remark}

\newtheorem{mybem}[mydef]{Remark}
\newtheorem{myex}[mydef]{Example}

\theoremstyle{plain}

\newtheorem{mycol}[mydef]{Corollary}
\newtheorem{mysen}[mydef]{Theorem}
\newtheorem{mylem}[mydef]{Lemma}

\newtheorem{myfact}[mydef]{Fact}
\newtheorem{myclaim}{Claim}
\newtheorem{mysenx}{Theorem}

\numberwithin{mydef}{section}

\DeclareMathOperator{\cof}{cof}
\DeclareMathOperator{\dom}{dom}

\DeclareMathOperator{\otp}{otp}

\DeclareMathOperator{\Add}{Add}

\DeclareMathOperator{\Coll}{Coll}
\DeclareMathOperator{\Odd}{Odd}

\DeclareMathOperator{\DSS}{DSS}

\DeclareMathOperator{\Lim}{Lim}

\DeclareMathOperator{\IA}{IA}
\DeclareMathOperator{\ISNIC}{ISNIC}

\newcommand{\dM}{\mathbb{M}}
\newcommand{\dP}{\mathbb{P}}
\newcommand{\dQ}{\mathbb{Q}}

\newcommand{\dS}{\mathbb{S}}
\newcommand{\dT}{\mathbb{T}}

\newcommand{\uhr}{\upharpoonright}

\newcommand{\ZFC}{\mathsf{ZFC}}

\newcommand{\AP}{\mathsf{AP}}

\title{Disjoint Stationary Sequences on an Interval of Cardinals} 

\author[H. Jakob]{Hannes Jakob} 
\address{Mathematisches Institut\\ University of Freiburg \\
	Ernst-Zermelo-Stra{\ss}e 1, 79104 Freiburg}
\email{hannes.jakob@mathematik.uni-freiburg.de}

\subjclass[2020]{Primary: 03E05, Secondary: 03E35, 03E55} 

\date{\today}

\begin{document}
	
	
	\baselineskip=17pt
	
	\keywords{Disjoint stationary sequences, Approachability property, Mitchell forcing, Strong distributivity}
	
	
	\begin{abstract}
		We answer a question of Krueger by obtaining -- from countably many Mahlo cardinals -- a model where there is a disjoint stationary sequence on $\aleph_{n+2}$ for every $n\in\omega$. In that same model, the notions of being internally stationary and internally club are distinct on a stationary subset of $[H(\Theta)]^{\aleph_{n+1}}$ for every $n\in\omega$ and $\Theta\geq\aleph_{n+2}$, answering another of Krueger's questions. This is obtained by employing a product of variants of Mitchell forcing which uses finite support for the Cohen reals and full support for the countably many collapses.
	\end{abstract}
	
	\maketitle
	
	\section{Introduction}
	
	The notion of a \emph{disjoint club sequence on $\mu^+$} -- a sequence $(A_{\alpha})_{\alpha\in S}$ where $S\subseteq\mu^+\cap\cof(\mu)$ is stationary and each $A_{\alpha}$ is club in $[\alpha]^{<\mu}$ -- was isolated by Friedman and Krueger (see \cite{FriedmanKruegerThinDisjoint}) in order to show that there can consistently be a fat stationary subset of $\omega_2$ which cannot obtain a club subset in any forcing extension preserving both $\omega_1$ and $\omega_2$. Later on, Krueger (see \cite{KruegerApplicMSI}) defined the related notion of a \emph{disjoint stationary sequence} -- where instead of requiring each $A_{\alpha}$ to be club we merely require it to be stationary. The related principle $\DSS(\mu^+)$ (stating that there exists a disjoint stationary sequence on $\mu^+$) serves as a strengthening of $\neg\AP_{\mu}$ which has the advantage of being more easily preserved and has applications regarding the distinctions of variants of internal approachability, defined by Foreman and Todorcevic (see \cite{ForemanTodorLowenheimSkolem}).
	
	In \cite{KruegerApplicMSI}, Krueger shows using his method of \emph{mixed support iterations} that there can consistently be a disjoint stationary sequence on any successor of a regular cardinal. In that same paper, he asks whether there can consistently be disjoint stationary sequences on two (or even infinitely many) successive cardinals. The background of this question is as follows: Krueger's mixed support iterations are modeled after the poset Mitchell employed in \cite{MitchellTreeProp} to construct a model where the tree property holds at the double successor of any cardinal. Mitchell's result was later extended first by Abraham (see \cite{AbrahamTrees}) and then by Cummings and Foreman (see \cite{CumForeTreeProp}), who showed that an iteration of a guessing variant of Mitchell forcing can be used to obtain a model where $\aleph_{n+2}$ has the tree property for every $n\in\omega$.
	
	In \cite{LevineDisjointStatSeq}, Levine showed that a variant of Mitchell forcing also forces the existence of a disjoint stationary sequence and employed an iteration of two instances of his poset to construct a model where $\DSS(\aleph_2)$ and $\DSS(\aleph_3)$ hold simultaneously. However, extending this to infinitely many successive cardinals simultaneously has some technical problems (which we will elaborate on in Section 5) due to the differences between his forcing and Mitchell's original poset. In this paper, we will show that a product of instances of Levine's forcing is also sufficient to obtain $\DSS(\aleph_2)\wedge\DSS(\aleph_3)$. We will then take a product of infinitely many instances of his poset with finite supports on the Cohen reals and full support on the collapses in order to solve Krueger's question in full:
	
	\begin{mysenx}\label{Thm1}
		Assume $(\kappa_n)_{n\in\omega}$ is an increasing sequence of Mahlo cardinals. There is a forcing extension in which, for each $n\in\omega$, $\aleph_{n+2}=\kappa_n$ and there is a disjoint stationary sequence on $\kappa_n$.
	\end{mysenx}
	
	Levine noted that in his model of $\DSS(\aleph_2)\wedge\DSS(\aleph_3)$, the notions of internal stationarity and clubness are distinct for stationarily many $N\in[H(\aleph_2)]^{\aleph_1}$ and $N\in[H(\aleph_3)]^{\aleph_2}$. The same -- and even more -- is also true in our case (answering an additional question from \cite{KruegerApplicMSI}):
	
	\begin{mysenx}\label{Thm2}
		Assume $(\kappa_n)_{n\in\omega}$ is an increasing sequence of Mahlo cardinals. There is a forcing extension in which, for each $n\in\omega$, $\aleph_{n+2}=\kappa_n$ and whenever $\Theta\geq\kappa_n$, there are stationarily many $N\in[H(\Theta)]^{<\kappa_n}$ which are internally stationary but not internally club.
	\end{mysenx}
	
	As part of the proof of the theorem above, we also show that merely a Mahlo cardinal suffices for a distinction between internal stationarity and clubness for stationarily many $N\in[H(\Theta)]^{\mu}$, answering another question of Krueger.
	
	The paper is organized as follows: In Section 2, we review preliminary definitions and results. In Section 3, we introduce the concept of strong ${<}\,\kappa$-distributivity, a regularity property for forcings which lies between ${<}\,\kappa$-distributivity and $\kappa$-strategic closure. In Section 4, we review Levine's poset from \cite{LevineDisjointStatSeq} and state its properties. In Section 5, we prove our main theorems.
	
	\section{Preliminaries}
	
	We assume the reader is familiar with the basics of forcing and the usage of large cardinals. Good introductory material can be found in the textbooks by Jech (see \cite{JechSetTheory}) and Kunen (see \cite{KunenSetTheory}). We follow the convention that filters are up-closed, so that if $p\leq q$, $p$ forces more than $q$.
	
	The notion of a disjoint stationary sequence was defined by Krueger in \cite{KruegerApplicMSI}:
	
	\begin{mydef}
		Let $\mu$ be a regular cardinal. A \emph{disjoint stationary sequence on $\mu^+$} is a sequence $(\mathcal{S}_{\alpha})_{\alpha\in S}$ where
		\begin{enumerate}
			\item $S\subseteq\mu^+\cap\cof(\mu)$ is stationary,
			\item for any $\alpha\in S$, $\mathcal{S}_{\alpha}\subseteq[\alpha]^{<\mu}$ is stationary.
		\end{enumerate}
		We let $\DSS(\mu^+)$ state that there exists a disjoint stationary sequence on $\mu^+$.
	\end{mydef}
	
	The principle $\DSS(\mu^+)$ is related to the failure of the approachability property at $\mu^+$:
	
	\begin{mydef}
		Let $\mu$ be a cardinal. A set $S\subseteq\mu^+$ is in the \emph{approachability ideal $I[\mu^+]$} if there exists a sequence $(a_{\alpha})_{\alpha\in\mu^+}$ of elements of $[\mu^+]^{<\mu}$ and a club $C\subseteq\mu^+$ such that whenever $\gamma\in S\cap C$, there is $A\subseteq\gamma$ unbounded in $\gamma$ such that $A\cap\beta\in\{a_{\alpha}\;|\;\alpha<\gamma\}$ for any $\beta<\gamma$.
		
		The \emph{approachability property} $\AP_{\mu}$ states that $I[\mu^+]$ is improper, i.e. $\mu^+\in I[\mu^+]$.
	\end{mydef}
	
	Also related are the variants of internal approachability which were first defined by Foreman and Todorcevic (see \cite{ForemanTodorLowenheimSkolem}) in order to prove a strengthening of the well-known model-theoretic Löwenheim-Skolem theorem.
	
	\begin{mydef}
		Let $X$ be a set, $\mu$ a regular cardinal and $N\in[X]^{\mu}$. We say that
		\begin{enumerate}
			\item $N$ is internally unbounded if $[N]^{<\mu}\cap N$ is unbounded in $[N]^{<\mu}$,
			\item $N$ is internally stationary if $[N]^{<\mu}\cap N$ is stationary in $[N]^{<\mu}$,
			\item $N$ is internally club if $[N]^{<\mu}\cap N$ contains a club in $[N]^{<\mu}$ and
			\item $N$ is internally approachable if there exists a sequence $(x_i)_{i\in\mu}$ of elements of $[N]^{<\mu}$ such that $\bigcup_{i\in\mu}x_i=N$ and $(x_i)_{i<j}\in N$ for every $j<\mu$.
		\end{enumerate}
	\end{mydef}
	
	The previous concepts are related as follows: Krueger showed (see \cite[Corollary 3.7]{KruegerApplicMSI}) that whenever $(\mathcal{S}_{\alpha})_{\alpha\in S}$ is a disjoint stationary sequence on $\mu^+$, no stationary subset of $S$ is in $I[\mu^+]$. In particular, $\DSS(\mu^+)$ implies $\neg\AP_{\mu}$ (the upshot of this is that $\DSS(\mu^+)$ is in general more easily preserved than $\neg\AP_{\mu}$; this was e.g. exploited in \cite[Theorem 5.1]{HonzikStejskalovaCardInvariant}). He also showed (see \cite[Theorem 6.5]{KruegerApplicMSI}) that, assuming $2^{\mu}=\mu^+$, $\DSS(\mu^+)$ is equivalent to the existence of stationarily many $N\in[H(\mu^+)]^{\mu}$ which are internally unbounded but not internally club (it is shown in \cite{JakobCascadingVariants} that this equivalence relies on the assumption that $2^{\mu}=\mu^+$). Lastly, a folklore result (see e.g. \cite[Lemma 1]{CoxFAAppSSR}) states that, again assuming $2^{\mu}=\mu^+$, $\AP_{\mu}$ fails if and only if there are stationarily many $N\in[H(\mu^+)]^{\mu}$ which are internally unbounded but not internally approachable (this also relies on the assumption $2^{\mu}=\mu^+$).
	
	We will obtain our consistency results using variants of Mitchell forcing. A common technique when working with such posets is the use of a projection analysis:
	
	\begin{mydef}
		Let $\dP$ and $\dQ$ be forcing orders. A function $\pi\colon\dP\to\dQ$ is a \emph{projection} if the following hold:
		\begin{enumerate}
			\item $\pi(1_{\dP})=1_{\dQ}$.
			\item For all $p\leq q$, $\pi(p)\leq \pi(q)$
			\item For all $p\in\dP$, if $q\leq \pi(p)$, there is some $p'\leq p$ such that $\pi(p')\leq q$.
		\end{enumerate}
	\end{mydef}
	
	If there exists a projection from $\dP$ to $\dQ$, any extension by $\dQ$ can be forcing extended to an extension by $\dP$:
	
	\begin{mydef}
		Let $\dP$ and $\dQ$ be forcing orders, $\pi\colon\dP\to\dQ$ a projection. Let $H$ be $\dQ$-generic. In $V[H]$, the forcing order $\dP/H$ consists of all $p\in\dP$ such that $\pi(p)\in H$, ordered as a suborder of $\dP$. We let $\dP/\dQ$ be a $\dQ$-name for $\dP/\dot{H}$ and call $\dP/\dQ$ the \emph{quotient forcing} of $\dP$ and $\dQ$.
	\end{mydef}
	
	\begin{myfact}
		Let $\dP$ and $\dQ$ be forcing orders and $\pi\colon\dP\to\dQ$ a projection. If $H$ is $\dQ$-generic over $V$ and $G$ is $\dP/H$-generic over $V[H]$, then $G$ is $\dP$-generic over $V$ and $H\subseteq\pi[G]$. In particular, $V[H][G]=V[G]$.
	\end{myfact}
	
	The canonical projection for Mitchell forcing comes from the \emph{termspace forcing}, an idea due to Laver:
	
	\begin{mydef}
		Let $\dP$ be a poset and $\dot{\dQ}$ a $\dP$-name for a poset. The \emph{termspace order} $\leq^*$ is an order on the set $T(\dP,\dot{\dQ})$ of all $\dP$-names for elements of $\dQ$ given by $\dot{q}'\leq^*\dot{q}$ if and only if $1_{\dP}\Vdash\dot{q}'\leq_{\dQ}\dot{q}$.
	\end{mydef}
	
	Using standard arguments on names, one easily shows:
	
	\begin{mylem}
		Let $\dP$ be a poset and $\dot{\dQ}$ a $\dP$-name for a poset. The identity is a projection from $\dP\times T(\dP,\dot{\dQ})$ onto $\dP*\dot{\dQ}$.
	\end{mylem}
	
	We will be using Jech's notion of \emph{generalized stationarity}: For $\mu$ a regular cardinal and $X$ a set of size $\geq\mu$, $C\subseteq[X]^{<\mu}$ is \emph{club} if for any $y\in[X]^{<\mu}$ there is $c\in C$ with $y\subseteq c$ and $C$ is closed under ascending unions of size ${<}\,\mu$. A set is \emph{stationary} if it intersects every club. An important (but easy to prove) result is the following:
	
	\begin{mylem}[{\cite[Theorem 1.5]{MenasStrongCSuperC}}]
		A set $C\subseteq[X]^{<\mu}$ is club if and only if there is a function $f\colon[X]^{<\omega}\to[X]^{<\mu}$ such that whenever $x\in[X]^{<\mu}$ and $\bigcup_{y\in[x]^{<\omega}}f(y)\subseteq x$, $x\in C$.
	\end{mylem}
	
	This means that a set $S\subseteq[X]^{<\mu}$ is stationary if and only if whenever $f\colon[X]^{<\omega}\to[X]^{<\mu}$, there is $x\in S$ which is closed under $f$ (i.e. $f(y)\subseteq x$ for any $y\in[x]^{<\omega}$). Moreover, Menas' result implies that whenever $C\subseteq[X]^{<\mu}$ is club and $Y\subseteq X$ has size $\geq\mu$, $\{c\cap Y\;|\;c\in C\}$ contains a club in $[Y]^{<\mu}$. Another fact we will be using is the following (the proof is the same as the usual proof for the corresponding fact for club subsets of $\mu$):
	
	\begin{mylem}\label{StatPres}
		Let $\mu$ be a regular cardinal and $X$ a set of size $\geq\mu$. Let $\dP$ be a $\mu$-cc poset and $G$ be $\dP$-generic. If $C\subseteq[X]^{<\mu}$ is club in $V[G]$, there exists $D\subseteq C$ with $D\in V$ such that $D$ is club in $[X]^{<\mu}$. In particular, any stationary subset of $[X]^{<\mu}$ is stationary in $[X]^{<\mu}$ in $V[G]$.
	\end{mylem}
	
	We will later ``force over'' set-sized elementary submodels of large substructures of the universe. For this, we use the following notation: If $M$ is a set, $\dP$ is a poset and $G$ is a $\dP$-generic filter, $M[G]$ is the set of $\tau^G$ for all $\tau\in M$ which are $\dP$-names. We have the following statement which is proven similarly to the result that any ccc poset is proper:

	\begin{mylem}
		Let $\Theta$ be a cardinal, $M\prec H(\Theta)$ and $\dP\in M$ a poset. Assume that $\kappa$ is a cardinal such that $M\cap\kappa\in\kappa$ and $\dP$ is $\kappa$-cc. Let $G$ be $\dP$-generic. In $V[G]$, $M[G]\cap V=M$.
	\end{mylem}
	
	\begin{proof}
		It is clear that $M[G]\cap V\supseteq M$, since $M$ contains $\check{x}$ whenever $x\in M$.
		
		Let $\tau\in M$ be a $\dP$-name such that $\tau^G\in V$. By modifying $\tau$ if necessary (without changing its evaluation according to $G$) we can assume that $1_{\dP}\Vdash\tau\in\check{V}$. So since $\tau\in M$ and $M\prec H(\Theta)$, there is a maximal antichain $A\in M$ of conditions forcing $\tau=\check{x}$ for $x\in V$. Since $\dP$ is $\kappa$-cc, there is in $M$ a cardinal $\mu<\kappa$ and a surjection $f\colon\mu\to A$. However, as $\mu<\kappa$ and $M\cap\kappa\in\kappa$, $\mu\subseteq M$ and so $f[\mu]=A\subseteq M$. In particular, there is $p\in G\cap A\cap M$ forcing $\tau=\check{x}$ for $x\in V$. By elementarity, $x\in M$ and so $\tau^G=x\in M$.
	\end{proof}
	
	\section{Strongly Distributive Forcings}
	
	In this section, we introduce a strengthening of distributivity which axiomatizes a common technique when working with ${<}\,\kappa$-closed partial orders (the construction of a ``sufficiently generic'' sequence of length $\kappa$). By design, this property is able to replace ${<}\,\kappa$-closure in many applications (such as the preservation of the stationarity of subsets of $\kappa$ or the $\kappa$-cc of forcing notions) but has one crucial advantage: Unlike ${<}\,\kappa$-closure, strong ${<}\,\kappa$-distributivity is preserved by $\kappa$-cc forcing extensions. This will later be applied in order to show that the tails of the product we use to obtain our main theorem do not destroy the disjoint stationary sequences which we already added.
	
	\begin{mydef}
		Let $\dP$ be a poset and $\kappa$ a cardinal. $\dP$ is \emph{strongly ${<}\,\kappa$-distributive} if for any sequence $(D_{\alpha})_{\alpha<\kappa}$ of open dense subsets of $\dP$ and any $p\in\dP$, there is a descending sequence $(p_{\alpha})_{\alpha<\kappa}$ such that $p_0\leq p$ and $\forall\alpha<\kappa$, $p_{\alpha}\in D_{\alpha}$. Such a sequence will be called a \emph{thread through $(D_{\alpha})_{\alpha<\kappa}$}.
	\end{mydef}
	
	Strong ${<}\,\kappa$-distributivity can be thought of as having ${<}\,\kappa$-distributivity witnessed in a uniform way: If $(D_{\alpha})_{\alpha<\kappa}$ is a sequence of open dense subsets of some ${<}\,\kappa$-distributive forcing notion, there is a sequence $(p_{\alpha})_{\alpha<\kappa}$ such that for all $\alpha<\kappa$, $p_{\alpha}\leq p_0$ and $p_{\alpha}\in\bigcap_{\beta<\alpha}D_{\beta}$ (since the intersection of ${<}\,\kappa$ open dense sets is open dense). However, we cannot in general find such a sequence in a uniform way, i.e. such that it is descending.
	
	Obviously strong ${<}\,\kappa$-distributivity implies ${<}\,\kappa$-distributivity. Note that strong ${<}\,\kappa$-distributivity and ${<}\,\kappa$-distributivity are not equivalent: If $S\subseteq\omega_1$ is stationary, the usual forcing shooting a club through $S$ by initial segments is ${<}\,\omega_1$-distributive. However, if $\omega_1\smallsetminus S$ is also stationary, the poset necessarily destroys the stationarity of that set, so the forcing cannot be strongly ${<}\,\omega_1$-distributive by Lemma \ref{PreservationProp} below.
	
	Keeping with the theme of strong ${<}\,\kappa$-distributivity being a uniform version of ${<}\,\kappa$-distributivity, we have the following characterisation: Recall that for antichains $A$ and $B$ we say that $A$ \emph{refines} $B$ if for every $q\in A$ there is $q'\in B$ with $q\leq q'$.
	
	\begin{mylem}\label{StrongDistAntichain}
		For a forcing order $\dP$ and a cardinal $\kappa$, the following are equivalent:
		\begin{enumerate}
			\item $\dP$ is strongly ${<}\,\kappa$-distributive.
			\item $\dP$ is ${<}\,\kappa$-distributive and for $p\in\dP$ and any descending sequence $(A_{\alpha})_{\alpha<\kappa}$ (with regards to refinement) of maximal antichains below $p$, there is a descending sequence $(p_{\alpha})_{\alpha<\kappa}$ such that $p_0\leq p$ and for any $\alpha$, $p_{\alpha}\in A_{\alpha}$.
		\end{enumerate}
	\end{mylem}
	
	\begin{proof}
		Assume $\dP$ is strongly ${<}\,\kappa$-distributive. Of course, this implies that $\dP$ is ${<}\,\kappa$-distributive. Assume for notational simplicity that $p=1_{\dP}$. Let $(A_{\alpha})_{\alpha<\kappa}$ be a sequence of maximal antichains in $\dP$ such that for $\beta<\alpha$, $A_{\alpha}$ refines $A_{\beta}$. For $\alpha<\kappa$, let $D_{\alpha}$ be the downward closure of $A_{\alpha}$ and consider a thread $(q_{\alpha})_{\alpha<\kappa}$ through $(D_{\alpha})_{\alpha<\kappa}$. For any $\alpha<\kappa$, let $p_{\alpha}$ be the unique (by pairwise incompatibility) element of $A_{\alpha}$ that is above $q_{\alpha}$. We are done after showing
		\begin{myclaim}
			The sequence $(p_{\alpha})_{\alpha<\kappa}$ is descending.
		\end{myclaim}
		\begin{proof}
			Let $\beta<\alpha$ be arbitrary. Because $A_{\alpha}$ refines $A_{\beta}$, there exists $p_{\beta}'\in A_{\beta}$ such that $p_{\alpha}\leq p_{\beta}'$. Thus, $q_{\alpha}\leq p_{\alpha}\leq p_{\beta}'$ and $q_{\alpha}\leq q_{\beta}\leq p_{\beta}$. In summary, $p_{\beta}'$ and $p_{\beta}$ are compatible and therefore equal.
		\end{proof}
		The other direction is easy. Given a sequence $(D_{\alpha})_{\alpha<\kappa}$ of open dense subsets of $\dP$ we can use the assumed ${<}\,\kappa$-distributivity of $\dP$ to build a sequence $(A_{\alpha})_{\alpha<\kappa}$ such that for any $\alpha<\kappa$, $A_{\alpha}$ refines $A_{\beta}$ for all $\beta<\alpha$ and is contained in $D_{\alpha}$.
	\end{proof}
	
	While ${<}\,\kappa$-distributivity means that every ${<}\,\kappa$-sequence of ground-model elements is in the ground model, strong ${<}\,\kappa$-distributivity means that we can uniformly approximate $\kappa$-sequences of ground-model elements:
	
	\begin{mylem}\label{UniformDecision}
		Let $\dP$ be a poset and $\kappa$ a cardinal. If $\dP$ is strongly ${<}\,\kappa$-distributive, $p\in\dP$ and $\dot{f}$ is a $\dP$-name such that $p\Vdash\dot{f}:\check{\kappa}\longrightarrow V$, there is a descending sequence $(p_{\alpha})_{\alpha<\kappa}$ with $p_0\leq p$ such that for every $\alpha<\kappa$, $p_{\alpha}$ decides $\dot{f}(\check{\alpha})$.
	\end{mylem}
	
	\begin{proof}
		This is clear: Consider $D_{\alpha}:=\{q\in\dP\;|\;q\text{ decides }\dot{f}(\check{\alpha})\}$.
	\end{proof}
	
	As is the case for ${<}\,\kappa$-distributivity, the converse holds for separative forcing orders (but we will never use this).
	
	We can now prove the preservation results mentioned above:
	
	\begin{mylem}\label{PreservationProp}
		Let $\dP$ be a poset and $\kappa$ a regular cardinal. Assume that $\dP$ is strongly ${<}\,\kappa$-distributive. Then the following hold:
		\begin{enumerate}
			\item If $S\subseteq\kappa$ is stationary, it remains stationary after forcing with $\dP$.
			\item If $\dQ$ is a $\kappa$-cc poset, it remains $\kappa$-cc after forcing with $\dP$.
		\end{enumerate}
	\end{mylem}
	
	\begin{proof}
		The proofs of (1) and (2) are basically the same: Assume toward a contradiction that $\dot{f}$ is a $\dP$-name for an ascending enumeration of a club in $\kappa$ (resp. an enumeration of an antichain in $\dQ$), forced by some $p\in\dP$. By Lemma \ref{UniformDecision}, let $g$ be a function on $\kappa$ and $(p_{\alpha})_{\alpha<\kappa}$ a descending sequence in $\dP$ with $p_0\leq p$ such that $p_{\alpha}\Vdash\dot{f}(\check{\alpha})=\check{g}(\check{\alpha})$. It then follows that the image of $g$ is a club in $\kappa$ (resp. an antichain in $\dQ$). In (1), we can find $\alpha$ such that $g(\alpha)\in S$ which implies $p_{\alpha}\Vdash\dot{f}(\check{\alpha})\in\check{S}$ and in (2), we directly obtain a contradiction.
	\end{proof}
	
	Since clubs in $[X]^{<\mu}$ are basically the same as clubs in $\mu$ whenever $|X|=\mu$, we obtain:
	
	\begin{mycol}\label{PresPropGen}
		Let $\dP$ be a poset, $\kappa$ a regular cardinal and $X$ a set with size $\kappa$. If $\dP$ is strongly ${<}\,\kappa$-distributive and $S\subseteq[X]^{<\kappa}$ is stationary, $S$ remains stationary after forcing with $\dP$.
	\end{mycol}
	
	We will later obtain a converse to Lemma \ref{PreservationProp} (2) by showing that any strongly ${<}\,\kappa$-distributive poset remains strongly ${<}\,\kappa$-distributive after forcing with any $\kappa$-cc partial order. However, this proof requires us to build a sequence $(p_{\alpha})_{\alpha<\kappa}$ of elements of $\dP$ where $p_{\alpha}$ is not just any element of a pre-determined open subset of $\dP$ that lies below $(p_{\beta})_{\beta<\alpha}$ but actually depends on the previous choices. We will now show that this is indeed possible by relating strong distributivity to the \emph{completeness game} played on a partial order. An equivalent characterization for distributivity was found by Foreman in \cite{ForemanBool} with almost the same proof.
	
	\begin{mydef}
		Let $\dP$ be a forcing order, $\delta$ an ordinal. The \emph{completeness game} $G(\dP,\delta)$ on $\dP$ with length $\delta$ has players COM (complete) and INC (incomplete) playing elements of $\dP$ with COM playing at even ordinals (and limits) and INC playing at odd ordinals. COM starts by playing $1_{\dP}$, afterwards $p_{\alpha}$ has to be a lower bound of $(p_{\beta})_{\beta<\alpha}$. INC wins if COM is unable to play at some point $<\delta$. Otherwise, COM wins.
		
		We say that $\dP$ is \emph{${<}\,\delta$-strategically closed} if COM has a winning strategy in $G(\dP,\alpha)$ for any $\alpha<\delta$. We say that $\dP$ is \emph{$\delta$-strategically closed} if COM has a winning strategy in $G(\dP,\delta)$.
	\end{mydef}
	
	Foreman showed the following connection between distributivity and the completeness game (see \cite[Page 718]{ForemanBool}):
	
	\begin{mysen}\label{DistCompleteness}
		If $\kappa=\lambda^+$ is a successor, $\dP$ is ${<}\,\kappa$-distributive if and only if INC does not have a winning strategy in $G(\dP,\lambda+1)$.
	\end{mysen}
	
	If INC does not have a winning strategy in $G(\dP,\lambda+1)$, they also do not have a winning strategy in $G(\dP,\delta)$ for any $\delta<\lambda^+$. Having this witnessed uniformly suggests the following statement:
	
	\begin{mysen}\label{DistCompletenessStrong}
		Let $\dP$ be a poset and $\kappa$ a cardinal. $\dP$ is strongly ${<}\,\kappa$-distributive if and only if INC does not have a winning strategy in $G(\dP,\kappa)$.
	\end{mysen}
	
	This characterization is very powerful because it allows us to construct sequences $(p_{\alpha})_{\alpha<\kappa}$ where $p_{\alpha}$ actually depends on $(p_{\beta})_{\beta<\alpha}$ and is not just any lower bound of $(p_{\beta})_{\beta<\alpha}$ in some open dense set. This idea is e.g. used below in the proof of the strong Easton lemma. Another use is in \cite{JakobSlenderApprox}, where the assumption of strong ${<}\,\kappa$-distributivity replaces that of ${<}\,\kappa$-closure in a theorem regarding the approximation property for certain iterations.
	
	For the proof, we follow Foreman's argument in \cite{ForemanBool} while providing more details and avoiding the use of the boolean completion of $\dP$.
	
	\begin{proof}[Proof of Theorem \ref{DistCompletenessStrong}]
		Assume that $\dP$ is not strongly ${<}\,\kappa$-distributive and $(D_{\alpha})_{\alpha<\kappa}$ is a sequence of open dense sets without a thread below some $p\in\dP$. Let INC first play $p$ and then, after $(p_{\delta})_{\delta<\gamma+2n+1}$ has been played, a condition $p_{\gamma+2n+1}\in D_{\gamma+n}$ with $p_{\gamma+2n+1}\leq p_{\gamma+2n}$. It is clear that this strategy wins for INC (otherwise, a losing game for INC would give rise to a thread through $(D_{\alpha})_{\alpha<\kappa}$ by picking out the odd conditions).
		
		Now assume that $\sigma$ is a winning strategy for INC in $G(\dP,\kappa)$. Let $\sigma(1_{\dP})=p$. We will construct a sequence $(A_{\alpha})_{\alpha\in\kappa}$ by induction such that the following holds:
		\begin{enumerate}
			\item For each $\alpha\in\kappa$ and $p_{\alpha}\in A_{\alpha}$, there exists a unique descending sequence $(p_{\beta})_{\beta<\alpha}$ such that for all $\beta\leq\alpha$, $p_{\beta}\in A_{\beta}$ and if $\beta\leq\alpha$ is odd, $p_{\beta}=\sigma((p_{\delta})_{\delta<\beta})$.
			\item If $\alpha\in\kappa$ is odd, $A_{\alpha}$ is a maximal antichain below $p$ (for even $\alpha$, we carefully choose $A_{\alpha}$ to obtain uniqueness in (1)).
		\end{enumerate}
		It follows from (1) and (2) that the sequence $(A_{\alpha})_{\alpha\in\kappa\cap\Odd}$ is a sequence of maximal antichains below $p$ which is descending with regards to refinement.
		
		To begin, let $A_0:=\{1_{\dP}\}$ and $A_1:=\{p\}$. Assume that we have constructed $(A_{\alpha})_{\alpha<\gamma}$, where $\gamma$ is an even successor ordinal. We will construct $A_{\gamma}$ and $A_{\gamma+1}$ simultaneously. Let $D_{\gamma+1}$ consist of all $p'\in\dP$ such that there exists a sequence $(p_{\alpha})_{\alpha<\gamma+2}$ with $p_{\gamma+1}=p'$ such that for all $\alpha<\gamma$, $p_{\alpha}\in A_{\alpha}$ and if $\alpha$ is odd, $p_{\alpha}=\sigma((p_{\beta})_{\beta<\alpha})$.
		\setcounter{myclaim}{0}
		\begin{myclaim}
			$D_{\gamma+1}$ is dense below $p$.
		\end{myclaim}
		\begin{proof}
			Let $p''\leq p$ be arbitrary. By maximality of $A_{\gamma-1}$, there exists $p_{\gamma-1}\in A_{\gamma-1}$ compatible with $p''$, witnessed by some $p^*$. By the inductive hypothesis, there exists a unique sequence $\overline{p}=(p_{\beta})_{\beta<\gamma-1}$ with $p_{\beta}\in A_{\beta}$ for $\beta<\gamma-1$ and $p_{\beta}=\sigma((p_{\delta})_{\delta<\beta})$ for all odd $\beta\leq\gamma-1$. Hence, letting $s:=\overline{p}^{\frown}p_{\gamma-1}{}^{\frown}p^*$, $s^{\frown}\sigma(s)$ witnesses density, since $\sigma(s)\leq p^*\leq p''$.
		\end{proof}
		Let $A_{\gamma+1}\subseteq D_{\gamma+1}$ be a maximal antichain below $p$. For each $p_{\gamma+1}\in A_{\gamma+1}$, by the definition of $D_{\gamma+1}$, there exists a descending sequence $(p_{\alpha})_{\alpha<\gamma+1}$ such that for all $\alpha\leq\gamma+1$, $p_{\alpha}\in A_{\alpha}$ and if $\alpha\leq\gamma+1$ is odd, $p_{\alpha}=\sigma((p_{\beta})_{\beta<\alpha})$. Choose such a sequence for each $p_{\gamma+1}\in A_{\gamma+1}$ and let $A_{\gamma}$ consist of the $\gamma$th entries of these sequences.
		\begin{myclaim}
			For each $p_{\gamma+1}\in A_{\gamma+1}$, there exists a unique sequence $(p_{\beta})_{\beta<\gamma+1}$ such that for all $\beta\leq\gamma+1$, $p_{\beta}\in A_{\beta}$ and if $\beta\leq\gamma+1$ is odd, $p_{\beta}=\sigma((p_{\delta})_{\delta<\beta})$
		\end{myclaim}
		\begin{proof}
			Let $p_{\gamma+1}\in A_{\gamma+1}$ and let $s$ be the chosen sequence witnessing $p_{\gamma+1}\in D_{\gamma+1}$. We will verify that any sequence as above is equal to $s$. So let $s'=(p_{\beta}')_{\beta\leq\gamma+1}$ be a different descending sequence such that for all $\beta\leq\gamma+1$, $p_{\beta}'\in A_{\beta}$ and if $\beta\leq\gamma+1$ is odd, $p_{\beta}'=\sigma((p_{\delta}')_{\delta<\beta})$ with $p_{\gamma+1}=p_{\gamma+1}'$.
			
			It follows that $p_{\gamma-1}'$ and $p_{\gamma-1}$ are compatible (witnessed by $p_{\gamma+1}=p_{\gamma+1}'$) and thus equal, as $A_{\gamma-1}$ is an antichain. By the inductive hypothesis $s\uhr\gamma=s'\uhr\gamma$, so $p_{\gamma}\neq p_{\gamma}'$. Since $p_{\gamma}'\in A_{\gamma}$ there is $a_{\gamma+1}\in A_{\gamma+1}$ (necessarily different from $p_{\gamma+1}$) and a sequence $t$ witnessing $a_{\gamma+1}\in D_{\gamma+1}$ such that $t(\gamma)=p_{\gamma}'$. Then $t(\gamma-1)$ and $p_{\gamma-1}'$ are compatible (witnessed by $p_{\gamma}'=t(\gamma)$) and hence equal. As before, this implies $t\uhr\gamma=s'\uhr\gamma$. However, this means that
			$$\sigma((p_{\delta}')_{\delta<\gamma+1})=p_{\gamma+1}'=p_{\gamma+1}\neq a_{\gamma+1}=\sigma(t\uhr\gamma+1)$$
			contradicting the fact that $\sigma$ is a function and $(p_{\delta}')_{\delta<\gamma+1}=t\uhr\gamma+1$.
		\end{proof}
		Assume $\gamma$ is a limit. Let $A_{\gamma}'$ be a common refinement of $A_{\alpha}$ for odd $\alpha<\gamma$. Given $p\in A_{\gamma}'$, let $p_{\alpha}\in A_{\alpha}$ witness refinement for odd $\alpha$ and let $p_{\alpha}\in A_{\alpha}$ witness $p_{\alpha+1}\in A_{\alpha+1}$ for even $\alpha$. Then $(p_{\alpha})_{\alpha<\gamma}$ is a play according to $\sigma$ by uniqueness (which implies that the sequences witnessing $p_{\alpha}\in A_{\alpha}$ are coherent). Let $D_{\gamma}$ be the downward closure of $A_{\gamma}'$ and let $D_{\gamma+1}$ consist of $\sigma(s)$ for sequences $s=(p_{\alpha})_{\alpha<\gamma+1}$ with $p_{\gamma}\in D_{\gamma}$ and $s\uhr\gamma$ witnessing this. Thus, $D_{\gamma+1}$ is dense and we can proceed as in the previous step: Let $A_{\gamma+1}\subseteq D_{\gamma+1}$ be a maximal antichain. Let $A_{\gamma}\subseteq D_{\gamma}$ contain one witness to $p\in D_{\gamma+1}$ for each $p\in A_{\gamma+1}$. Then clearly for any $p\in A_{\gamma+1}$ a sequence as claimed exists. As before, if there exist two sequences $t,t'$ for one $p\in A_{\gamma+1}$, $t\uhr\gamma=t'\uhr\gamma$, since $t(\gamma)\in A_{\gamma}\subseteq A_{\gamma'}$ and thus lies below exactly one element of each $A_{\alpha}$ for odd $\alpha$.
		
		Lastly, by Lemma \ref{StrongDistAntichain}, if $\dP$ is strongly ${<}\,\kappa$-distributive, then there exists a thread through $(A_{\alpha})_{\alpha\in\kappa\cap\Odd}$, i.e. a sequence $(p_{\alpha})_{\alpha\in\kappa\cap\Odd}$ such that for odd $\alpha$, $p_{\alpha}\in A_{\alpha}$. For even $\alpha$, let $p_{\alpha}\in A_{\alpha}$ witness $p_{\alpha+1}\in A_{\alpha+1}$. By uniqueness, $(p_{\alpha})_{\alpha<\kappa}$ is a play in $G(\dP,\kappa)$ according to $\sigma$. But this contradicts our assumption that $\sigma$ was a winning strategy.
	\end{proof}
	
	Since clearly at most one player can have a winning strategy in any game, we have the following:
	
	\begin{mylem}
		Let $\dP$ be a poset and $\kappa$ a cardinal. If $\dP$ is $\kappa$-strategically closed, then $\dP$ is strongly ${<}\,\kappa$-distributive.
	\end{mylem}
	
	We will later give an example of a poset which is strongly ${<}\,\kappa$-distributive but not $\kappa$-strategically closed.
	
	The main point for introducing strong ${<}\,\kappa$-distributivity is the following strengthening of the Easton lemma. The second statement in the Lemma was also noticed (in a different form) by Andreas Lietz on Mathoverflow after a question by the author (see \cite{LietzAnswer}).
	
	\begin{mylem}\label{StrongEaston}
		Let $\kappa$ be a regular cardinal. Assume $\dQ$ is $\kappa$-cc and $\dP$ is strongly ${<}\,\kappa$-distributive.
		\begin{enumerate}
			\item $1_{\dP}\Vdash\check{\dQ}\text{ is }\check{\kappa}\text{-cc.}$
			\item $1_{\dQ}\Vdash\check{\dP}\text{ is strongly}<\check{\kappa}\text{-distributive}$.
		\end{enumerate}
	\end{mylem}
	
	\begin{proof}
		Part (1) was shown in Lemma \ref{PreservationProp} (2) so we show (2).
		\setcounter{myclaim}{0}
		We first prove the following helpful claim:
		\begin{myclaim}
			If $D\subseteq\dQ\times\dP$ is open dense and $q\in\dQ$, the set $D_q$ consisting of all $p\in\dP$ such that for some $A\subseteq\dQ$ that is a maximal antichain below $q$, $A\times\{p\}\subseteq D$, is open dense in $\dP$.
		\end{myclaim}
		\begin{proof}
			Openness is clear: If $A$ witnesses $p\in D_q$ and $p'\leq p$, $A$ also witnesses $p'\in D_q$.
			
			Thus, assume the set is not dense and there is $p\in\dP$ such that for every $p'\leq p$, $p'\notin D_q$. We will give a winning strategy for INC in $G(\dQ,\kappa)$. We will assume that whenever $(p_{\alpha})_{\alpha<\gamma}$ has been played according to our strategy, we have constructed an antichain $\{q_{\alpha}\;|\;\alpha\in\gamma\cap\Odd\}$ below $q$ such that for any $\alpha<\gamma$, $(q_{\alpha},p_{\alpha})\in D$. To begin, let INC find a pair $(q_1,p_1)\leq(q,p)$ with $(q_1,p_1)\in D$ and play $p_1$.
				
			Assume the game has lasted until $\gamma$, $\gamma+1$ is odd and the position is $(p_{\alpha})_{\alpha<\gamma+1}$. If $\{q_{\alpha}\;|\;\alpha\in\gamma\cap\Odd\}$ is maximal, it witnesses $p_{\gamma}\in D$ by the openness of $D$: For every $\alpha\in\gamma\cap\Odd$, $(q_{\alpha},p_{\alpha})\in D$ and thus $(q_{\alpha},p_{\gamma})\in D$. This contradicts our assumption, since $p_{\gamma}\leq p$. It follows that there exists some $q_{\gamma+1}'$ which is incompatible with every $q_{\alpha}$. By open density, there exists $(q_{\gamma+1},p_{\gamma+1})\leq(q_{\gamma+1}',p_{\gamma})$, $(q_{\gamma+1},p_{\gamma+1})\in D$. Let INC play $p_{\gamma+1}$.
			
			This strategy is a winning strategy, because a play of length $\kappa$ would give us a $\kappa$-sized antichain in $\dQ$, contradicting the $\kappa$-cc of $\dQ$. However, the existence of this winning strategy contradicts Theorem \ref{DistCompletenessStrong}. So the claim is true.
		\end{proof}
		Now assume $\dot{f}$ and $\tau$ are $\dQ$-names such that some $q\in\dQ$ forces $\dot{f}$ to map $\check{\kappa}$ to open dense subsets of $\dP$ and $\tau$ to be an element of $\dP$. Strengthening $q$ if necessary, we can assume $q\Vdash\tau=\check{p}$ for some $p\in\dP$.
		\begin{myclaim}
			Each set $D_{\alpha}:=\{(q',p')\in\dQ\uhr q\times\dP\;|\;q'\Vdash\check{p}'\in\dot{f}(\check{\alpha})\}$ is open dense.
		\end{myclaim}
		\begin{proof}
			Openness in both coordinates follows either from the properties of the forcing relation or from $\dot{f}(\check{\alpha})$ being forced by $q$ to be open.
				
			For density, let $(q',p')\in\dQ\uhr q\times\dP$ be arbitrary. Thus $q'\Vdash\exists\tau(\tau\in\dot{f}(\check{\alpha})\wedge\tau\leq \check{p}')$. Because $\tau$ is in particular forced to be in $V$, there exists $q''\leq q'$ and $p''$ such that
			$$q''\Vdash(\check{p}''\in\dot{f}(\check{\alpha})\wedge \check{p}''\leq\check{p}')$$
			Thus, $(q'',p'')\leq(q',p')$ and $(q'',p'')\in D_{\alpha}$
		\end{proof}
		Combining the two claims, for each $\alpha$, the set $D_{\alpha}'$ of all $p'\in\dP$ such that for some $A\subseteq\dQ$ that is a maximal antichain below $q$ we have $A\times\{p'\}\subseteq D_{\alpha}$ is open dense in $\dP$. If $p'\in D_{\alpha}'$, there exists a maximal antichain $A$ below $q$ such that for every $q'\in A$, $q'\Vdash\check{p}'\in\dot{f}(\check{\alpha})$. By maximality, $q\Vdash\check{p}'\in\dot{f}(\check{\alpha})$.
			
		Let $(p_{\alpha})_{\alpha<\kappa}$ be a thread through $(D_{\alpha}')_{\alpha<\kappa}$ below $p$. Then $q$ forces $(\check{p}_{\alpha})_{\alpha<\kappa}$ to be a thread through $\dot{f}$ below $\check{p}$.
	\end{proof}
	
	In particular, if $\dP$ is ${<}\,\kappa$-closed and $\dQ$ is $\kappa$-cc., $\dP$ is strongly ${<}\,\kappa$-distributive after forcing with $\dQ$.
	
	We close this section by providing two examples of strongly distributive forcings.
	
	\begin{myex}
		For $S$ a stationary subset of $[\omega_2]^{<\omega_1}$, let $\dP(S)$ be the forcing consisting of increasing and continuous functions $p\colon\alpha\to S$ where $\alpha<\omega_1$ is a successor ordinal. $\dP(S)$ collapses $\omega_2$ to $\omega_1$ by adding a cofinal sequence $(x_{\alpha})_{\alpha<\omega_1}$ of elements of $S$. Krueger showed (see \cite[Proposition 2.2]{KruegerMitchellStyle}) that the term ordering on $\Add(\omega)*\dP([\omega_2]^{<\omega_1}\cap V)$ is $\omega_1$-strategically closed. This shows that $\dP([\omega_2]^{<\omega_1}\cap V)$ is strongly ${<}\,\omega_1$-distributive in $V[G]$ where $G$ is $\Add(\omega)$-generic. However, if $H$ is now $\Coll(\omega_1,\omega_2)$-generic over $V[G]$, $\dP([\omega_2]^{<\omega_1}\cap V)$ is no longer strongly ${<}\,\omega_1$-distributive because it destroys the stationarity of the stationary set $[\omega_2^V]^{<\omega_1}\smallsetminus V$ even though $|\omega_2^V|=\omega_1$ in $V[G][H]$ (see Fact \ref{KruegerFact}).
	\end{myex}
	
	This example shows the following:
	
	\begin{enumerate}
		\item A strongly ${<}\,\omega_1$-distributive forcing need not be proper: In $V[G]$, the forcing $\dP([\omega_2]^{<\omega_1}\cap V)$ destroys the stationarity of the stationary set $[\omega_2]^{<\omega_1}\smallsetminus V$ (again, see Fact \ref{KruegerFact}).
		\item A strongly ${<}\,\omega_1$-distributive forcing need not remain strongly ${<}\,\omega_1$-distributive in a countably closed forcing extension.
		\item A strongly ${<}\,\omega_1$-distributive forcing need not be ${<}\,\omega_1$-strategically closed.
	\end{enumerate}
	
	\begin{myex}\label{Example2}
		Let $\kappa$ be a regular cardinal and $\dS(\kappa)$ the poset adding a $\square_{\kappa}$-sequence by initial segments (see e.g. \cite[Example 6.6]{CummingsHandbook}): $p\in\dS(\kappa)$ if and only if:
		\begin{enumerate}
			\item $\dom(p)=\beta+1\cap\Lim$ for some limit ordinal $\beta\in\kappa^+$.
			\item For $\alpha\in\dom(p)$, $p(\alpha)$ is club in $\alpha$ and $\otp(p(\alpha))\leq\kappa$.
			\item If $\alpha\in\dom(p)$ then for all $\beta\in\lim(p(\alpha))$, $p(\alpha)\cap\beta= p(\beta)$.
		\end{enumerate}
		we order $\dS(\kappa)$ by end-extension.
		
		$\dS(\kappa)$ is ${<}\,\kappa^+$-strategically closed. However, $\dS(\kappa)$ is not strongly ${<}\,\kappa^+$-distributive if $\square_{\kappa}$ fails: For $\alpha<\kappa^+$ a limit ordinal, let $D_{\alpha}$ consist of all $p\in\dS(\kappa)$ with $\alpha\in\dom(p)$ (let $D_{\alpha}:=\dS(\kappa)$ if $\alpha$ is not a limit). Each $D_{\alpha}$ is open dense, but given any descending sequence $(p_{\alpha})_{\alpha<\kappa^+}$ of elements of $\dS(\kappa)$ with $p_{\alpha}\in D_{\alpha}$ for every $\alpha\in\kappa^+$, $\bigcup_{\alpha<\kappa^+}p_{\alpha}$ is a $\square_{\kappa}$-sequence, a contradiction.
	\end{myex}
	
	Combined, these two examples show that there is no provable relationship between ${<}\,\kappa^+$-strategic closure and strong ${<}\,\kappa^+$-distributivity. However, assuming $\square_{\kappa}$ holds, Ishiu and Yoshinobu showed (see \cite[Theorem 3.3]{IshiuYoshinobuDirectiveTreesGamesPosets}) that any ${<}\,\kappa^+$-strategically closed poset is also $\kappa^+$-strategically closed and thus strongly ${<}\,\kappa^+$-distributive. So in some sense Example \ref{Example2} is canonical.
	
	\section{The Mitchell Forcing}

	In this section we will introduce the forcing that was used by Levine to obtain disjoint stationary sequences on $\aleph_2$ and $\aleph_3$ simultaneously.
	
	To motivate his definition, we first state the following fact due to Krueger which is crucial for our arguments:
	
	\begin{mydef}
		Let $\tau\leq\mu\leq\Theta$ be cardinals and $N\in[H(\Theta)]^{<\mu}$. Then $N$ is \emph{weakly internally approachable of length $\tau$} if $N=\bigcup_{i<\tau}N_i$, where $(N_i)_{i<j}\in N$ for every $j<\tau$. We let $\IA(\tau)$ be the collection of all $N$ (for any $\Theta$ and $\mu$) which are weakly internally approachable of length $\tau$.
	\end{mydef}
	
	\begin{myfact}[{\cite[Theorem 7.1]{KruegerApplicMSI}}]\label{KruegerFact}
		Suppose $V\subseteq W$ are models of $\ZFC$ with the same ordinals and there is a real in $W\smallsetminus V$. Let $\mu$ be a regular uncountable cardinal in $W$ and let $X$ be a set in $V$ such that $(\mu^+)^W\subseteq X$. In $W$ let $\Theta\geq\mu^+$ be a regular cardinal such that $X\subseteq H(\Theta)$. Then in $W$ the collection of all $N\in[H(\Theta)]^{<\mu}\cap\IA(\omega)$ with $N\cap X\notin V$ is stationary.
	\end{myfact}
	
	The upshot of Fact \ref{KruegerFact} is that stationary sets consisting of weakly internally approachable models are preserved by sufficiently closed forcings:
	
	\begin{myfact}[{\cite[Lemma 2.2 and 2.4]{KruegerApplicMSI}}]\label{IAPreserve}
		Let $\tau<\mu$ be cardinals and $\Theta\geq\mu$. If $S\subseteq[H(\Theta)]^{<\mu}\cap\IA(\tau)$ is stationary and $\dP$ is a ${<}\,\mu$-closed poset, $\dP$ forces that $S$ remains stationary in $[H^V(\Theta)]^{<\mu}$.
	\end{myfact}
	
	Krueger's idea to obtain a disjoint stationary sequence on some cardinal $\mu^+$ can be summarized as follows: After forcing with $\Add(\omega)$ -- thanks to Fact \ref{KruegerFact} -- there are stationarily many $N\in[H(\Theta)]^{<\mu}\cap\IA(\omega)$ with $N\cap\mu^+\notin V$. By the weak internal approachability this stationarity is preserved when forcing with the Levy-collapse $\Coll(\mu,\Theta)$. And so by iterating forcings of the form $\Add(\omega)*\dot{\Coll}(\check{\mu},\check{\mu}^+)$ with Mahlo length one obtains a disjoint stationary sequence on a stationary subset $S\subseteq\mu^+$ (consisting of the previously inaccessible cardinals) by choosing $\mathcal{S}_{\alpha}$ to consist of precisely those elements of $[\alpha]^{<\mu}$ which were added at stage $\alpha+1$.
	
	Krueger took care of such an iteration by using his method of \emph{mixed support iterations} (see \cite{KruegerMitchellStyle}). Levine noted that the projection analysis provided through the use of Mitchell forcing leads to better preservation properties which allowed him to obtain disjoint stationary sequences on $\aleph_2$ and $\aleph_3$ simultaneously.
	
	For a regular cardinal $\tau$ and a set $Y$ we let $\Add^*(\tau,Y)$ consist of all functions $p\colon\{\delta\in Y\;|\;\delta\text{ is inaccessible}\}\times\tau\to\{0,1\}$, ordered by $\supseteq$.
	
	\begin{mydef}
		Let $\lambda$ be inaccessible and let $\tau<\mu<\lambda$ be regular cardinals such that $\tau^{<\tau}=\tau$. We let $\dM^+(\tau,\mu,\lambda)$ consist of pairs $(p,q)$ such that
		\begin{enumerate}
			\item $p\in\Add^*(\tau,\lambda)$.
			\item $q$ is a function such that
			\begin{enumerate}
				\item $\dom(q)$ is a ${<}\,\mu$-sized set such that for each $\delta\in\dom(q)$ $\delta=\nu+1$ for an inaccessible cardinal $\nu<\lambda$,
				\item whenever $\delta\in\dom(q)$, $q(\delta)$ is an $\Add^*(\tau,\delta+1)$-name for a condition in $\dot{\Coll}(\check{\mu},\check{\delta})$.
			\end{enumerate}
		\end{enumerate}
		We let $(p',q')\leq(p,q)$ if and only if
		\begin{enumerate}
			\item $p'\leq p$ in $\Add^*(\tau,\lambda)$,
			\item $\dom(q')\supseteq\dom(q)$ and whenever $\delta\in\dom(q)$,
			$$p'\uhr((\delta+1)\times\tau)\Vdash q'(\check{\delta})\leq q(\check{\delta})$$
		\end{enumerate}
	\end{mydef}
	
	This forcing is quite similar to the original poset Mitchell used to obtain the tree property at an arbitrary double successor cardinal (see \cite{MitchellTreeProp}). One crucial difference is the ``de-coupling'' of $\tau$ and $\mu$ (otherwise $\mu$ is often fixed as $\tau^+$). This is necessary because there is no currently known analogue of Fact \ref{KruegerFact} which works for $\Add(\tau)$ when $\tau>\omega$ (see \cite[Question 12.6]{KruegerApplicMSI}), so we have to add reals if we want to obtain $\DSS$.
	
	One important definition regarding variants of Mitchell forcing is the \emph{term ordering} from which many regularity properties are derived:
	
	\begin{mydef}
		Let $\lambda$ be inaccessible and let $\tau<\mu<\lambda$ be regular cardinals such that $\tau^{<\tau}=\tau$. We let $\dT(\tau,\mu,\lambda)$ consist of all $(p,q)\in\dM^+(\tau,\mu,\lambda)$ such that $p=\emptyset$, ordered as a suborder of $\dM^+(\tau,\mu,\lambda)$.
	\end{mydef}
	
	We now collect some facts about $\dM^+$ and $\dT$ in the following ``omnibus lemma'' (the proofs are found in \cite[Section 1.3]{LevineDisjointStatSeq}):
	
	\begin{mylem}\label{OmnibusLemma}
		Let $\lambda$ be inaccessible and let $\tau<\mu<\lambda$ be regular cardinals such that $\tau^{<\tau}=\tau$.
		\begin{enumerate}
			\item $\dT(\tau,\mu,\lambda)$ is ${<}\,\mu$-closed.
			\item There is a projection from $\Add^*(\tau,\lambda)\times\dT(\tau,\mu,\lambda)$ onto $\dM^+(\tau,\mu,\lambda)$.
			\item $\dM^+(\tau,\mu,\lambda)$ is ${<}\,\tau$-closed, $\lambda$-cc and preserves all cardinals up to and including $\mu$ as well as above and including $\lambda$.
			\item $\dM^+(\tau,\mu,\lambda)$ forces $2^{\tau}=2^{\mu}=\mu^+=\lambda$.
		\end{enumerate}
	\end{mylem}
	
	A concept we want to highlight specifically is the following: If $\nu<\lambda$ is inaccessible and $G$ is $\dM^+(\tau,\mu,\nu)$-generic, the quotient forcing $\dM^+(\tau,\mu,\lambda)/G$ (using the obvious projection) once again resembles Mitchell forcing and is in particular the projection of the product of a $\tau^+$-cc and a ${<}\,\mu$-closed forcing. Because of our specific order of Cohen reals and collapses (using $\Add(\tau)$ at inaccessibles and $\Coll(\mu,\delta)$ at successors of inaccessibles), we obtain the following:
	
	\begin{mylem}[{\cite[Lemma 22]{LevineDisjointStatSeq}}]\label{Decomp}
		Let $\lambda$ be inaccessible and let $\tau<\mu<\lambda$ be regular cardinals such that $\tau^{<\tau}=\tau$. If $\nu<\lambda$ is inaccessible, there is a forcing equivalence
		$$\dM^+(\tau,\mu,\lambda)\cong\dM^+(\tau,\mu,\nu)*\Add(\tau)*\Omega$$
		where $\dM^+(\tau,\mu,\nu)*\Add(\tau)$ forces that $\Omega$ is the projection of the product of a $\tau^+$-cc and a ${<}\,\mu$-closed poset.
	\end{mylem}
	
	Building on arguments of Krueger, Levine was able to show that his variant of Mitchell forcing also forces the existence of a disjoint stationary sequence:
	
	\begin{mylem}\label{MForcesDSS}
		Let $\lambda$ be inaccessible and let $\mu<\lambda$ be a regular cardinal. If $\lambda$ is Mahlo then $\dM^+(\omega,\mu,\lambda)$ forces $\DSS(\mu^+)$.
	\end{mylem}
	
	\section{Proving the Main Theorems}
	
	Now we will set up the forcing which gives us our desired model. Fix an increasing sequence $(\kappa_n)_{n\in\omega}$ of Mahlo cardinals. For simplicity, let $\kappa_{-1}:=\aleph_1$.
	
	\begin{mydef}
		Let $\dP((\kappa_n)_{n\in\omega})$ be the following poset: Conditions are functions $p$ on $\omega$ such that
		\begin{enumerate}
			\item For any $n\in\omega$, $p(n)=(p_0(n),p_1(n))\in\dM^+(\omega,\kappa_{n-1},\kappa_n)$
			\item For all but finitely many $n\in\omega$, $p_0(n)=\emptyset$.
		\end{enumerate}
		
		We let $p'\leq p$ if and only if $p'(n)\leq p(n)$ for all $n\in\omega$.
		
		For $k\in\omega$, we let $\dP_k((\kappa_n)_{n\in\omega})$ consist of those $p\in\dP((\kappa_n)_{n\in\omega})$ such that $p(n)$ is trivial for $k\geq n$. Let $\dP^k((\kappa_n)_{n\in\omega})$ consist of those $p\in\dP((\kappa_n)_{n\in\omega})$ such that $p(n)$ is trivial for $k<n$. Both posets inherit their ordering from $\dP((\kappa_n)_{n\in\omega})$.
	\end{mydef}
	
	Clearly, $\dP((\kappa_n)_{n\in\omega})$ is isomorphic to the product $\dP_k((\kappa_n)_{n\in\omega})$ and $\dP^k((\kappa_n)_{n\in\omega})$ whenever $k\in\omega$. Moreover, $\dP_k((\kappa_n)_{n\in\omega})$ is isomorphic to the simple ``normal product'' of $\dM^+(\omega,\kappa_{n-1},\kappa_n)$ over $n<k$. So clearly $\dP((\kappa_n)_{n\in\omega})$ is isomorphic to $\dP_k((\kappa_n)_{n\in\omega})\times\dM^+(\omega,\kappa_{n-1},\kappa_n)\times\dP^k((\kappa_n)_{n\in\omega})$. To simplify notation later on, we let $\dP((\kappa_n)_{n\in\omega})(n):=\dM^+(\omega,\kappa_{n-1},\kappa_n)$.
	
	The reason for introducing this specific product is as follows: Because of a lack of an analogue of Fact \ref{KruegerFact} we have to use $\dM^+(\omega,\kappa_{n-1},\kappa_n)$ to force $\DSS(\kappa_{n-1}^+)$ and cannot (as in \cite{CumForeTreeProp} for the tree property) use $\dM^+(\kappa_{n-2},\kappa_{n-1},\kappa_n)$. So none of our forcings are even countably closed. Due to this, a full support product cannot be expected to preserve cardinals. On the other hand, a finite support product would lack a sufficiently closed term ordering in order to guarantee the preservation of smaller cardinals.
	
	This poset is quite similar to the forcing used by Unger in \cite{UngerSuccessiveApproach} (before Lemma 3.1 there). In that work, the poset is used to obtain a failure of approachability at all successor cardinals in the interval $[\aleph_2,\aleph_{\omega^2+3}]$. In his case, the idea of adding many Cohen subsets to a small cardinal is used in order to obtain the failure of the approachability property at double successors of singular cardinals: The failure of $\AP_{\mu}$ implies $2^{<\mu}\geq\mu^+$ and so having $\neg\AP_{\aleph_{\omega+1}}$ implies $2^{\aleph_{\omega}}\geq\aleph_{\omega+2}$. This means that either $\aleph_{\omega}$ is not a strong limit (this is true in Unger's model and the reason why he adds many Cohen subsets to, in his case, $\aleph_1$) or the singular cardinal hypothesis has to fail at $\aleph_{\omega}$. However, in Unger's model $\AP_{\aleph_{\omega}}$ fails as well and it is to this day unknown whether both the approachability property and the singular cardinal hypothesis can fail simultaneously at $\aleph_{\omega}$.
	
	Let us also note that due to the easier preservation of $\DSS$ (or the distinction between internal stationarity and clubness) when compared to the tree property, we are fortunate to be able to use a product instead of an iteration which simplifies some arguments.
	
	We will now obtain a similar projection analysis as in Lemma \ref{OmnibusLemma}. As for $\dP$ before, we let $\dT_k((\kappa_n)_{n\in\omega}):=\prod_{n<k}\dT(\omega,\kappa_{n-1},\kappa_n)$, $\dT((\kappa_n)_{n\in\omega})(k):=\dT(\omega,\kappa_{k-1},\kappa_k)$ and $\dT^k((\kappa_n)_{n\in\omega}):=\prod_{n\geq k}\dT(\omega,\kappa_{n-1},\kappa_n)$.
	
	\begin{mylem}\label{PDecomp}
		Let $k\in\omega$.
		\begin{enumerate}
			\item $\dP_k((\kappa_n)_{n\in\omega})$ is $\kappa_{k-1}$-Knaster.
			\item $\dP^k((\kappa_n)_{n\in\omega})$ is the projection of $\prod_{n\geq k}\Add^*(\omega,\kappa_n)\times\dT^k((\kappa_n)_{n\in\omega})$, where the product is taken with finite support and $\dT^k((\kappa_n)_{n\in\omega})$ is ${<}\,\kappa_{k-1}$-closed. Moreover, the quotient $(\prod_{n\geq k}\Add^*(\omega,\kappa_n)\times\dT^k)/\dP^k((\kappa_n)_{n\in\omega})$ is forced to be ${<}\,\kappa_{k-1}$-distributive.
		\end{enumerate}
	\end{mylem}
	
	\begin{proof}
		(1) is easy, since $\dP_k((\kappa_n)_{n\in\omega})$ is simply a finite product of $\kappa_{k-1}$-Knaster forcings.
		
		For (2), given $n\geq k$, let $\pi_n\colon\Add^*(\omega,\kappa_n)\times\dT(\omega,\kappa_{n-1},\kappa_n)\to\dM^+(\omega,\kappa_{n-1},\kappa_n)$ be the projection obtained through Lemma \ref{OmnibusLemma} (2). Let $\dQ^k$ consist of all $(p,q)$ such that both $p$ and $q$ are functions on $\omega\smallsetminus k$, for each $n\in\omega\smallsetminus k$, $(p(n),q(n))\in\Add^*(\omega,\kappa_n)\times\dT(\omega,\kappa_{n-1},\kappa_n)$ and $p(n)$ is trivial for cofinitely many $n$, ordered pointwise. Then clearly $\pi\colon\dQ^k\to\dP^k((\kappa_n)_{n\in\omega})$, defined by $\pi((p,q))(n)=\pi_n(p(n),q(n))$, is a projection from $\dQ^k$ onto $\dP^k((\kappa_n)_{n\in\omega})$. Moreover, $\dQ^k$ is easily seen to be isomorphic to $\prod_{n\geq k}\Add^*(\omega,\kappa_n)\times\prod_{n\geq k}\dT(\omega,\kappa_{n-1},\kappa_n)$, where the first product is taken with finite support and the second product is taken with full support. Additionally, $\prod_{n\geq k}\dT(\omega,\kappa_{n-1},\kappa_n)=\dT^k((\kappa_n)_{n\in\omega})$.
		
		Since $\dT^k$ is a full support product of ${<}\,\kappa_{k-1}$-closed forcing notions, it is ${<}\,\kappa_{k-1}$-closed. The distributivity of the quotient follows easily from Easton's lemma, which implies that any ${<}\,\kappa_{k-1}$-sequence of ordinals added by $\prod_{n\geq k}\Add^*(\omega,\kappa_n)\times\dT^k$ has been added by $\prod_{n\geq k}\Add^*(\omega,\kappa_n)$ which is contained in $\dP^k((\kappa_n)_{n\in\omega})$.
	\end{proof}
	
	\subsection{Disjoint Stationary Sequences}\hfill
	
	We can now prove Theorem \ref{Thm1} easily:
	
	\begin{mysen}
		After forcing with $\dP((\kappa_n)_{n\in\omega})$, for every $n\in\omega$, $\kappa_n=\aleph_{n+2}$ and $\DSS(\aleph_{n+2})$ holds.
	\end{mysen}
	
	\begin{proof}
		We first show that for any $n\in\omega$, $\dP:=\dP((\kappa_n)_{n\in\omega})$ forces $\kappa_n=\kappa_{n-1}^+$. Let $G$ be $\dP$-generic over $V$. For $n\in\omega$, let $G(n)$ be the $\dM^+(\omega,\kappa_{n-1},\kappa_n)$-generic filter induced by $G$ and let $G_n$ (resp. $G^n$) be the $\dP_n$- (resp. $\dP^n$-) generic filter induced by $G$. By the product lemma, $V[G]=V[G(n)][G_n][G^{n+1}]$, since $\dP$ is isomorphic to $\dP_n\times\dM^+(\omega,\kappa_{n-1},\kappa_n)\times\dP^{n+1}$. In $V[G(n)]$, $\kappa_n=\kappa_{n-1}^+$ by Lemma \ref{OmnibusLemma} (4). Furthermore, $\dP_n$ is $\kappa_{n-1}$-Knaster in $V[G(n)]$, since it is $\kappa_{n-1}$-Knaster in $V$ (see Lemma \ref{PDecomp} (1)) and $\dM^+(\omega,\kappa_{n-1},\kappa_n)$ can be projected onto from the product of a ccc and a ${<}\,\kappa_{n-1}$-closed poset (see Lemma \ref{OmnibusLemma} (1) and (2)). Ergo $\kappa_n$ is preserved when going to $V[G(n)][G_n]$. Lastly, in $V$, $\dP^{n+1}$ can be projected onto from $\prod_{k\geq n+1}\Add^*(\omega,\kappa_k)\times\dT^{n+1}$, where $\dT^{n+1}$ is ${<}\,\kappa_n$-closed (see Lemma \ref{PDecomp} (2)). The same projection clearly works in $V[G(n)][G_n]$ as well. Moreover, in $V[G(n)][G_n]$, $\prod_{k\geq n+1}\Add^*(\omega,\kappa_k)$ is still ccc and $\dT^{n+1}$ is strongly ${<}\,\kappa_n$-distributive by Lemma \ref{StrongEaston} (since $\dP_{n+1}$ is $\kappa_n$-cc). Ergo $\kappa_n$ is preserved when going to $V[G(n)][G_n][G^{n+1}]$. It follows easily by induction that in $V[G]$, $\kappa_n=\aleph_{n+2}$ for every $n\in\omega$ (recall that $\kappa_{-1}=\aleph_1$).
		
		Now we turn to $\DSS$ which is proved almost exactly as above. For $n\in\omega$, we know by Lemma \ref{MForcesDSS} that $\DSS(\kappa_{n-1}^+)$ holds in $V[G(n)]$, witnessed by some sequence $(\mathcal{S}_{\alpha})_{\alpha\in S}$, where $S\subseteq\kappa_{n-1}^+\cap\cof(\kappa_{n-1})$ is stationary and each $\mathcal{S}_{\alpha}\subseteq[\alpha]^{<\kappa_{n-1}}$ is stationary. In $V[G(n)][G_n]$, $S$ and each $\mathcal{S}_{\alpha}$ remains stationary by the $\kappa_{n-1}$-cc of $\dP_n$. Lastly, $V[G(n)][G_n][G^{n+1}]$ is an extension of $V[G(n)][G_n]$ using a poset which can be projected onto from the product of a ccc and a strongly ${<}\,\kappa_n=\kappa_{n-1}^+$-distributive poset. By Lemma \ref{PreservationProp} and Corollary \ref{PresPropGen}, this poset also cannot destroy the stationarity of $S$ or any $\mathcal{S}_{\alpha}$. So $(\mathcal{S}_{\alpha})_{\alpha\in S}$ remains a disjoint stationary sequence on $\kappa_{n-1}^+$ in $V[G]$.
	\end{proof}
	
	\subsection{Distinction Between Internal Stationarity and Clubness}\hfill
	
	We now turn to the proof of Theorem \ref{Thm2}. Note that even for a distinction between internal stationarity and clubness in $[H(\aleph_n)]^{<\aleph_n}$, we have some work to do: In our final model, $2^{\omega}\geq\aleph_{\omega+1}$ and under this cardinal arithmetic the result by Krueger relating $\DSS(\aleph_n)$ and the distinction of internal stationarity and clubness in $[H(\aleph_n)]^{<\aleph_n}$ might fail (at least one direction can consistently fail, see \cite{JakobCascadingVariants}).
	
	We start by showing that $\dM^+(\omega,\mu,\lambda)$ forces the distinction between internal stationarity and clubness in a particularly strong way which is more easily preserved. This refines \cite[Lemma 29]{LevineDisjointStatSeq} but is proved with similar arguments.
	
	\begin{mysen}\label{BetterDist}
		Let $\aleph_0<\mu<\lambda$ be regular cardinals such that $\lambda$ is Mahlo. After forcing with $\dM^+(\omega,\mu,\lambda)$, for any $\Theta\geq\lambda$ there are stationarily many $N\in[H(\Theta)]^{\mu}$ such that $N$ is internally stationary but $N$ does not contain a club in $[N\cap\lambda]^{<\mu}$
	\end{mysen}
	
	Note that this directly implies that any such $N$ is not internally club: If $\{N_i\;|\;i\in\mu\}$ were club in $[N]^{<\mu}$ and contained in $N$, $\{N_i\cap\lambda\;|\;i\in\mu\}$ would contain a club in $[N\cap\lambda]^{<\mu}$ and be contained in $N$, since we can assume without loss of generality that $N$ is elementary in $H(\Theta)$ and $\lambda\in N$.
	
	\begin{proof}[Proof of Theorem \ref{BetterDist}]
		Write $\dM:=\dM^+(\omega,\mu,\lambda)$. Let $G$ be $\dM$-generic. In $V[G]$, let $F\colon[H^{V[G]}(\Theta)]^{<\omega}\to[H^{V[G]}(\Theta)]^{\mu}$. We aim to find $N\in[H^{V[G]}(\Theta)]^{\mu}$ which is closed under $F$ and internally stationary but not internally club. Let $\dot{F}$ be a name such that $\dot{F}^G=F$. In $V$, apply the Mahloness of $\lambda$ (see e.g. \cite{HarringtonShelahExactEquiconsistencies}) to find $M\prec H^V(\Theta')$ (for $\Theta'$ so large that $\dot{F}\in H^V(\Theta')$) such that
		\begin{enumerate}
			\item $M$ contains $\dot{F},\dM,\mu,\lambda,\Theta$,
			\item $\nu:=M\cap\lambda=|M|$ is an inaccessible cardinal below $\lambda$,
			\item $[M]^{<\nu}\subseteq M$.
		\end{enumerate}
		
		It follows that $N:=M[G]\cap H^{V[G]}(\Theta)\in[H^{V[G]}(\Theta)]^{\mu}$ is closed under $\dot{F}^G$. So all that is left to show is that $N$ is internally stationary but not internally club. Let $G'$ be the $\dM(\omega,\mu,\nu)$-generic filter induced by $G$. We note that $M[G]\cap V=M$ by the $\lambda$-cc of $\dM$.
		\setcounter{myclaim}{0}
		\begin{myclaim}\label{Claim1}
			If $x\in[M]^{<\nu}\cap V[G']$, $x\in M[G]$.
		\end{myclaim}
		
		\begin{proof}
			Let $\tau$ be an $\dM^+(\omega,\mu,\nu)$-name for a subset of $M$. By the $\nu$-cc of $\dM^+(\omega,\mu,\nu)$ and since $\dM^+(\omega,\mu,\nu)\subseteq M$ we can code $\tau$ as a ${<}\,\nu$-sized subset of $M$ which therefore is an element of $M$. Ergo $\tau^G\in M[G]$.
		\end{proof}
		
		In particular, $M[G]$ is closed under ${<}\,\nu$-sequences of ordinals lying in $V[G']$. Furthermore, $M[G]$ contains a bijection $\iota$ between some cardinal $\Delta$ and $H^{V[G]}(\Theta)$. This bijection restricts to one between $M[G]\cap\Delta=M\cap\Delta$ and $M[G]\cap H^{V[G]}(\Theta)$. We first show:
		
		\begin{myclaim}\label{Claim2}
			$[M\cap\Delta]^{<\mu}\cap M[G]$ is stationary in $[M\cap\Delta]^{<\mu}$ in $V[G]$.
		\end{myclaim}
		
		\begin{proof}
			By Claim \ref{Claim1}, $[M\cap\Delta]^{<\mu}\cap M[G]\supseteq[M\cap\Delta]^{<\mu}\cap V[G']$, so it suffices to show that the latter set is stationary in $[M\cap\Delta]^{<\mu}$. We first claim that in $V[G'']$ (where $G''$ is the $\dM^+(\omega,\mu,\nu)*\Add(\omega)$-generic filter induced by $G$) the set
			$$S:=\{X\in[H^{V[G'']}(\Delta)]^{<\mu}\;|\;X\in\IA(\omega)\wedge X\cap(M\cap\Delta)\in V[G']\}$$
			is stationary in $[H^{V[G'']}(\Delta)]^{<\mu}$. Whenever $F\colon[H^{V[G'']}(\Delta)]^{<\omega}\to[H^{V[G'']}(\Delta)]^{<\mu}$ and $\dot{F}$ is an $\Add(\omega)$-name for $F$ in $V[G']$, we can find $X\prec H^{V[G']}(\Delta')$ ($X\in V[G']$) for $\Delta'$ large enough with $\dot{F}\in X$ such that $X\in\IA(\omega)$. It follows that $X[H]\cap H^{V[G'']}(\Delta)$ (where $H$ is $\Add(\omega)$-generic with $G''=G'*H$) is closed under $F$, $X[H]\cap H^{V[G'']}(\Delta)\in\IA(\omega)$ (as witnessed by $(X_i[H]\cap H^{V[G'']}(\Delta))_{i\in\omega}$, where $(X_i)_{i\in\omega}$ witnesses $X\in\IA(\omega)$) and $X[H]\cap H^{V[G'']}(\Delta)\cap(M\cap\Delta)=X\cap(M\cap\Delta)\in V[G']$ by the ccc of $\Add(\omega)$.
			
			Since $V[G]$ is an extension of $V[G'']$ using $\Omega$, a poset which can be projected onto from the product of a ccc and a ${<}\,\mu$-closed poset (see Lemma \ref{Decomp}), $S$ is still stationary in $[H^{V[G'']}(\Delta)]^{<\mu}$ in $V[G]$ (by Fact \ref{IAPreserve}) from which it follows that $\{X\cap(M\cap\Delta)\;|\;X\in S\}$ is stationary in $[M\cap\Delta]^{<\mu}$ in $V[G]$. Since that set is contained in $V[G']$, we are done.
		\end{proof}
		
		This readily implies:
		
		\begin{myclaim}
			$N$ is internally stationary.
		\end{myclaim}
		
		\begin{proof}
			Recall that $\iota\uhr(M\cap\Delta)$ is a bijection between $M\cap\Delta$ and $M[G]\cap H^{V[G]}(\Theta)$ (which equals $N$). Let $c\subseteq[N]^{<\mu}$ be club. Then $\{\iota^{-1}[x]\;|\;x\in c\}$ is club in $[M\cap\Delta]^{<\mu}$ and so by Claim \ref{Claim2} there is $x\in c$ with $\iota^{-1}[x]\in M[G]$. However, since $\iota\in M$, $x=\iota[\iota^{-1}[x]]$ is in $M[G]$ as well and is in $H^{V[G]}(\Theta)$ by its size. So $N\cap c$ is nonempty.
		\end{proof}
		
		Now we show that $N$ is not ``internally club in $\lambda$'', i.e. there is no club in $[N\cap\lambda]^{<\mu}$ which is contained in $N$. We first show the following converse to Claim \ref{Claim1}:
		
		\begin{myclaim}
			If $x\in [N\cap\lambda]^{<\mu}\cap N$, $x\in V[G']$.
		\end{myclaim}
		
		\begin{proof}
			Since $|x|<\mu$, $x$ has been added by $\Add^*(\omega,\lambda)$ by Lemma \ref{OmnibusLemma} (2). Let $\tau$ be an $\Add^*(\omega,\lambda)$-name for $x$, $\tau\in M$. We can assume by the ccc of $\Add^*(\omega,\lambda)$ that $|\tau|<\mu$. It follows that $\tau\subseteq M$, so $\tau$ is an $\Add^*(\omega,M\cap\lambda)$-name. Ergo $x\in V[G']$.
		\end{proof}
		
		So assume that $c$ is club in $[N\cap\lambda]^{<\mu}=[\nu]^{<\mu}$. We will show that there is $x\in c$ with $x\notin V[G']$ which directly implies that $c$ is not contained in $N$.
		
		By Fact \ref{KruegerFact} (with $V[G']$ in lieu of $V$, $V[G'']$ in lieu of $W$ and $\nu$ in lieu of $X$ as well as $\Theta$), in $V[G'']$ there are stationarily many $X\in[H^{V[G'']}(\nu)]^{<\mu}\cap\IA(\omega)$ with $X\cap\nu\notin V[G']$. It follows as before that that same set is still stationary in $[H^{V[G'']}(\nu)]^{<\mu}$ in $V[G]$. So clearly there is such an $X$ with $X\cap\nu\in c$. However, this directly implies $X\cap\nu\notin V[G']$.
	\end{proof}
	
	\begin{mybem}
		As we are only using a Mahlo cardinal to obtain a distinction between internal clubness and approachability for stationarily many $N\in[H(\Theta)]^{\mu}$ for any $\Theta\geq\mu^+$, this theorem directly resolves a case of \cite[Question 12.7]{KruegerApplicMSI}.
	\end{mybem}
	
	We now prove two preservation theorems which will finish the proof of Theorem \ref{Thm2}. For simpler notation, we let $\ISNIC^+(\mu^+)$ state that for every $\Theta\geq\mu^+$ there are stationarily many $N\in[H(\Theta)]^{\mu}$ which are internally stationary but do not contain a club in $[N\cap\mu^+]^{<\mu}$. In this notation, Theorem \ref{BetterDist} states that $\dM^+(\omega,\mu,\lambda)$ forces $\ISNIC^+(\mu^+)$ whenever $\lambda$ is Mahlo.
	
	\begin{mylem}\label{PreserveUp}
		Let $\mu>\aleph_0$ be a regular cardinal and $\dQ$ a poset. Assume that $\ISNIC^+(\mu^+)$ holds and $\dQ$ is $\mu$-cc. Then $\dQ$ forces $\ISNIC^+(\mu^+)$.
	\end{mylem}
	
	\begin{proof}
		Let $G$ be $\dQ$-generic. In $V[G]$, let $F\colon[H^{V[G]}(\Theta)]^{<\omega}\to[H^{V[G]}(\Theta)]^{\mu}$ be any function and let $\dot{F}$ be a $\dQ$-name for $F$. In $V$, let $\Theta'$ be so large that $\dot{F},\dQ\in H^V(\Theta')$ and find $M\prec H^V(\Theta')$ of size $\mu$ with $M\cap\mu^+$ an ordinal such that $\dot{F},\dQ,\Theta\in M$ and $M$ is internally stationary but does not contain a club in $[M\cap\mu^+]^{<\mu}$. By the $\mu$-cc of $\dQ$, $M[G]\cap V=M$ and $M$ remains internally stationary in $V[G]$ (by Lemma \ref{StatPres}). Clearly $N:=M[G]\cap H^{V[G]}(\Theta)$ is closed under $F$.
		\setcounter{myclaim}{0}
		\begin{myclaim}
			$N$ is internally stationary.
		\end{myclaim}
		
		\begin{proof}
			Let $c\subseteq[N]^{<\mu}$ be club. Then the set of all $m\in[M]^{<\mu}$ with $m\prec M$ such that $m[G]\cap H^{V[G]}(\Theta)\in c$ is club in $[M]^{<\mu}$ by a simple interleaving argument. Ergo there exists $m\in M$ such that $m\in[M]^{<\mu}$, $m\prec M$ and $m[G]\cap H^{V[G]}(\Theta)\in c$. But $m[G]\cap H^{V[G]}(\Theta)$ is in $M[G]$ by elementarity and in $H^{V[G]}(\Theta)$ by its size, so it is in $N$.
		\end{proof}
		
		We are finished after showing:
		
		\begin{myclaim}
			$N$ does not contain a club in $[N\cap\mu^+]^{<\mu}$.
		\end{myclaim}
		
		\begin{proof}
			Let $c\subseteq[N\cap\mu^+]^{<\mu}$ be club in $[N\cap\mu^+]^{<\mu}$. Since $N\cap\mu^+=M\cap\mu^+$ (which is a ground-model set) and $\dQ$ has the $\mu$-cc, we can assume that $c\in V$ (by Lemma \ref{StatPres}). By assumption, there is $n\in c$ which is not in $M$ (but it is in $V$ since $c\in V$). But then $n$ is not in $M[G]$ either, since $M[G]\cap V=M$.
		\end{proof}
		
		This finishes the proof.
	\end{proof}
	
	Lastly, we have to show a downward preservation theorem, showing that sufficiently distributive forcings do not force $\ISNIC^+(\mu^+)$ if it does not already hold in the ground model.
	
	\begin{mylem}\label{PreserveDown}
		Let $\mu>\aleph_0$ be a regular cardinal and $\dQ$ a poset. Assume that $\ISNIC^+(\mu^+)$ holds in $V[G]$, where $G$ is $\dQ$-generic, and $\dQ$ is ${<}\,\mu^+$-distributive. Then $\ISNIC^+(\mu^+)$ holds in $V$.
	\end{mylem}
	
	\begin{proof}
		In $V$, let $\Theta$ be arbitrary and let $F\colon[H^V(\Theta)]^{<\omega}\to[H^V(\Theta)]^{\mu}$ be any function. Let $\Theta'$ be so large that $F,\dQ\in H(\Theta')$. The following is clear:
		\setcounter{myclaim}{0}
		\begin{myclaim}
			In $V[G]$, the set
			$$D:=\{M\in[H^V(\Theta')]^{\mu}\;|\;M[G]\cap V=M\}$$
			is club in $[H^V(\Theta')]^{\mu}$.
		\end{myclaim}
		
		Since $H^V(\Theta')[G]=H^{V[G]}(\Theta')$ by the assumption $\dQ\in H(\Theta')$, we know that $\{M[G]\;|\;M\in D\}$ is club in $[H^{V[G]}(\Theta')]^{\mu}$. So by $\ISNIC^+(\mu^+)$ holding in $V[G]$, we can find $M\in D$ such that $F\in M$, $H^V(\Theta)\in M$, $M[G]\prec(H^{V[G]}(\Theta'),\in)$ and $M[G]$ is internally stationary but does not contain a club in $[M[G]\cap\mu^+]^{<\mu}$. Clearly $N:=M[G]\cap H^V(\Theta)=M\cap H^V(\Theta)$ is closed under $F$, since $F\in V$. Also $N\in V$ by the distributivity of $\dQ$.
		
		\begin{myclaim}
			$N$ is internally stationary in $V$.
		\end{myclaim}
		
		\begin{proof}
			In $V$, let $c\subseteq[N]^{<\mu}$ be club. Then $\{m\in[M]^{<\mu}\;|\;m\cap H^V(\Theta)\in c\}$ is club in $[M]^{<\mu}$ and so $\{m[G]\;|\;m\in[M]^{<\mu}\wedge m\cap H^V(\Theta)\in c\}$ is club in $[M[G]]^{<\mu}$. Ergo there is $m\in[M]^{<\mu}$ with $m\cap H^V(\Theta)\in c$ such that $m[G]\in M[G]$ since $M[G]$ is internally stationary. As in the first claim, we may assume $m[G]\cap V=m$ from which it follows that
			$$m\cap H^V(\Theta)=m[G]\cap H^V(\Theta)\in c\cap M[G]\cap H^V(\Theta)=c\cap N$$
			since $m\cap H^V(\Theta)\in H^V(\Theta)$ by its size and the distributivity of $\dQ$.
		\end{proof}
		
		We are finished after showing:
		
		\begin{myclaim}
			$N$ does not contain a club in $[N\cap\mu^+]^{<\mu}$ in $V$.
		\end{myclaim}
		
		\begin{proof}
			Let $c\subseteq[N\cap\mu^+]^{<\mu}$ be club. Since $N\cap\mu^+=M[G]\cap\mu^+$, $c\subseteq N$ would imply that $M[G]\supseteq N$ contains a club in $[M[G]\cap\mu^+]^{<\mu}$, a clear contradiction.
		\end{proof}
		
		This finishes the proof of Lemma \ref{PreserveDown}.
	\end{proof}
	
	Finally, we can show Theorem \ref{Thm2}:
	
	\begin{mysen}
		After forcing with $\dP((\kappa_n)_{n\in\omega})$, for any $n\in\omega$ and $\Theta\geq\aleph_{n+2}$, there are stationarily many $N\in[H(\Theta)]^{\aleph_{n+1}}$ which are internally stationary but not internally club.
	\end{mysen}
	
	\begin{proof}
		Write $\dP:=\dP((\kappa_n)_{n\in\omega})$ and let $n\in\omega$ be arbitrary. We can regard forcing with $\dP$ as forcing first with $\dP^{n+1}$, followed by $\dM^+(\omega,\kappa_{n-1},\kappa_n)$ and lastly with $\dP_n$. Moreover, any extension by $\dP^{n+1}$ can be extended to an extension by $\prod_{k\geq n+1}\Add^*(\omega,\kappa_k)\times\dT^{n+1}$ using a ${<}\,\kappa_n$-distributive forcing (by Lemma \ref{PDecomp}). So in summary, any extension by $\dP$ can be extended to an extension by
		$$\dP^{\text{proj}}:=\dT^{n+1}\times\dM^+(\omega,\kappa_{n-1},\kappa_n)\times\prod_{k\geq n+1}\Add^*(\omega,\kappa_k)\times\dP_n$$
		using a ${<}\,\kappa_n=\kappa_{n-1}^+$-distributive forcing. In light of Lemma \ref{PreserveDown} it suffices to show that $\dP^{\text{proj}}$ forces $\ISNIC^+(\kappa_{n-1}^+)$ (since $\kappa_{n-1}$ becomes $\aleph_{n+1}$).
		
		Since $\dT^{n+1}$ is ${<}\,\kappa_n$-closed, it preserves the Mahloness of $\kappa_n$ and does not change the definition of $\dM^+(\omega,\kappa_{n-1},\kappa_n)$. By Lemma \ref{BetterDist}, $\dT^{n+1}\times\dM^+(\omega,\kappa_{n-1},\kappa_n)$ forces $\ISNIC^+(\kappa_{n-1}^+)$. In any extension by $\dT^{n+1}\times\dM^+(\omega,\kappa_{n-1},\kappa_n)$, the tail of the product $\dP^{\text{proj}}$ remains $\kappa_{n-1}$-cc and thus preserves $\ISNIC^+(\kappa_{n-1}^+)$ by Lemma \ref{PreserveUp}. Ergo $\dP^{\text{proj}}$ forces $\ISNIC^+(\kappa_{n-1}^+)$, so $\dP$ forces $\ISNIC^+(\kappa_{n-1}^+)$.
	\end{proof}
	
	\section{Open Questions}
	
	We close with a few open questions. As for $\DSS$, one technical aspect of this paper is that, in order to obtain $\DSS$, we had to add many Cohen reals. This is problematic because it means that there is no hope of obtaining a model where $\DSS$ holds at all successors of regular cardinals. Thus we ask (see also \cite[Question 12.3]{KruegerApplicMSI}):
	
	\begin{myque}
		Is it consistent there is a disjoint stationary sequence on $\aleph_3$ and $2^{\aleph_0}=\aleph_1$?
	\end{myque}
	
	This is related to obtaining a higher analogue of Fact \ref{KruegerFact} which works without adding reals. There is a motivating result by Dobrinen and Friedman (see \cite{DobrinenFriedmanCoStat}) who showed that it is consistent that the ground model is costationary after forcing with $\Add(\mu)$ for any regular $\mu$. However, they do not obtain stationarily many new structures \emph{which are weakly internally approachable of length $\mu$}, so it is unclear if their methods can be adapted to obtain a disjoint stationary sequence without adding reals.
	
	We are also interested in technical questions regarding strong distributivity. For one, we do not even know if there must always be a poset which is strongly distributive but does not have any kind of closure (all of our examples come from viewing old posets in generic extensions). Thus we ask:
	
	\begin{myque}
		Let $\kappa$ be a cardinal. Is it provable in $\ZFC$ that there is a poset $\dP$ which is strongly ${<}\,\kappa$-distributive but not $\kappa$-strategically closed?
	\end{myque}
	
	In other words, we are asking if it is consistent that the completeness game $G(\dP,\kappa)$ is determined for any partial order $\dP$.
	
	\subsection*{Acknowledgments}
	
	The author wants to thank Heike Mildenberger and Maxwell Levine for many helpful conversations regarding this work. He also wants to thank the anonymous referee for their helpful comments leading to improvement of the manuscript.
	
	\bibliographystyle{plain}

\begin{thebibliography}{10}
		
		\bibitem{AbrahamTrees}
		Uri Abraham.
		\newblock Aronszajn trees on $\aleph_2$ and $\aleph_3$.
		\newblock {\em Ann. Pure Appl. Logic}, 24(3):213--230, 1983.
		
		\bibitem{CoxFAAppSSR}
		Sean~D. Cox.
		\newblock Forcing axioms, approachability, and stationary set reflection.
		\newblock {\em J. Symb. Log.}, 86(2):499–530, 2021.
		
		\bibitem{CummingsHandbook}
		James Cummings.
		\newblock {\em Iterated Forcing and Elementary Embeddings}, pages 775--883.
		\newblock Springer Netherlands, Dordrecht, 2010.
		
		\bibitem{CumForeTreeProp}
		James Cummings and Matthew Foreman.
		\newblock The tree property.
		\newblock {\em Adv. Math.}, 133(1):1--32, 1998.
		
		\bibitem{DobrinenFriedmanCoStat}
		Natasha Dobrinen and Sy-David Friedman.
		\newblock Co-stationarity of the ground model.
		\newblock {\em J. Symb. Log.}, 71(3):1029--1043, 2006.
		
		\bibitem{ForemanBool}
		Matthew Foreman.
		\newblock Games played on boolean algebras.
		\newblock {\em J. Symb. Log.}, 48(3):714--723, 1983.
		
		\bibitem{ForemanTodorLowenheimSkolem}
		Matthew Foreman and Stevo Todorcevic.
		\newblock A new l{\"o}wenheim-skolem theorem.
		\newblock {\em Trans. Amer. Math. Soc.}, 357(5):1693--1715, 2005.
		
		\bibitem{FriedmanKruegerThinDisjoint}
		Sy-David Friedman and John Krueger.
		\newblock Thin stationary sets and disjoint club sequences.
		\newblock {\em Trans. Amer. Math. Soc.}, 359(5):2407--2420, 2007.
		
		\bibitem{HarringtonShelahExactEquiconsistencies}
		Leo Harrington and Saharon Shelah.
		\newblock {Some exact equiconsistency results in set theory.}
		\newblock {\em Notre Dame J. Form. Log.}, 26(2):178 -- 188, 1985.
		
		\bibitem{HonzikStejskalovaCardInvariant}
		Radek Honzik and Sarka Stejskalova.
		\newblock Generalized cardinal invariants for an inaccessible $\kappa$ with
		compactness at $\kappa^{++}$.
		\newblock {\em Arch. Math. Logic}, page to appear, 2025.
		
		\bibitem{IshiuYoshinobuDirectiveTreesGamesPosets}
		Tetsuya Ishiu and Yasuo Yoshinobu.
		\newblock Directive trees and games on posets.
		\newblock {\em Proc. Amer. Math. Soc.}, 130(5):1477--1485, 2002.
		
		\bibitem{JakobCascadingVariants}
		Hannes Jakob.
		\newblock Cascading variants of internal approachability.
		\newblock 2024.
		
		\bibitem{JakobSlenderApprox}
		Hannes Jakob.
		\newblock Slender trees and the approximation property.
		\newblock 2024.
		
		\bibitem{JechSetTheory}
		Thomas Jech.
		\newblock {\em Set Theory: The Third Millennium Edition}.
		\newblock Springer, 2003.
		
		\bibitem{KruegerMitchellStyle}
		John Krueger.
		\newblock A general mitchell style iteration.
		\newblock {\em MLQ Math. Log. Q.}, 54(6):641--651, 2008.
		
		\bibitem{KruegerApplicMSI}
		John Krueger.
		\newblock Some applications of mixed support iterations.
		\newblock {\em Ann. Pure Appl. Logic}, 158(1):40--57, 2009.
		
		\bibitem{KunenSetTheory}
		K.~Kunen.
		\newblock {\em Set Theory}.
		\newblock Studies in Logic: Mathematical. College Publications, 2011.
		
		\bibitem{LevineDisjointStatSeq}
		Maxwell Levine.
		\newblock On disjoint stationary sequences.
		\newblock {\em Pacific J. Math.}, 332(1):147--165, 2024.
		
		\bibitem{LietzAnswer}
		Andreas Lietz.
		\newblock Are the completeness games $g_{\lambda+1}(p)$ and $g_{\lambda^+}(p)$
		equivalent for inc?
		\newblock MathOverflow.
		
		\bibitem{MenasStrongCSuperC}
		Telis~K. Menas.
		\newblock On strong compactness and supercompactness.
		\newblock {\em Ann. Math. Logic}, 7(4):327--359, 1975.
		
		\bibitem{MitchellTreeProp}
		William Mitchell.
		\newblock Aronszajn trees and the independence of the transfer property.
		\newblock {\em Ann. Math. Logic}, 5(1):21--46, 1972.
		
		\bibitem{UngerSuccessiveApproach}
		Spencer Unger.
		\newblock Successive failures of approachability.
		\newblock {\em Israel J. Math.}, 242:663 -- 695, 2017.
		
	\end{thebibliography}

\end{document}